\documentclass{article}

\usepackage{amsmath, amssymb, amsthm, color}
\usepackage{cite}
\usepackage{authblk}

\newtheorem{theorem}{Theorem}
\newtheorem{defn}[theorem]{Definition}

\newtheorem{corl}[theorem]{Corollary}

\newtheorem*{theorem*}{Theorem}

\numberwithin{theorem}{section}
\numberwithin{equation}{section}

\newcommand{\dz}{\partial_z}
\newcommand{\disk}{\mathbb{D}}
\newcommand{\zbar}{\overline{z}}
\newcommand{\dzbar}{\partial_{\zbar}}

\newcommand{\diverg}{\operatorname{div}}
\newcommand{\imag}{\operatorname{Im}}
\newcommand{\real}{\operatorname{Re}}

\title{Improved H\"older Continuity of Quasiconformal Maps}
\author{Tyler Bongers 
	\thanks{Email: \texttt{charlesb@math.msu.edu} }}

\date{\today}

\begin{document}

\maketitle 

\begin{abstract}
Quasiconformal maps in the complex plane are homeomorphisms that satisfy certain geometric distortion inequalities; infinitesimally, they map circles to ellipses with bounded eccentricity. The local distortion properties of these maps give rise to a certain degree of global regularity and H\"older continuity. In this paper, we give improved lower bounds for the H\"older continuity of these maps; the analysis is based on combining the isoperimetric inequality with a study of the length of quasicircles. Furthermore, the extremizers for H\"older continuity are characterized, and some applications are given to solutions to elliptic partial differential equations.
\end{abstract}

\tableofcontents

\section{Introduction}
A $K$-quasiconformal map $f$ is an orientation-preserving homeomorphism between two domains in the complex plane, lying in the Sobolev space $W^{1, 2}_{\text{loc}}$ and satisfying the distortion inequality 
$$\max_{\beta} |\partial_{\beta} f| \le K \min_{\beta} |\partial_{\beta} f|$$ 
for almost all $z$, where $\partial_{\beta}$ is the directional derivative in the direction $\beta$. These maps can be realized as solutions to the Beltrami equation 
\begin{equation}
\partial_{\overline{z}} f = \mu(z) \partial_z f \label{eq:c_linear_beltrami}
\end{equation}
where the Beltrami coefficient $\mu(z)$ has the bound $\|\mu\|_{\infty} \le \frac{K - 1}{K + 1} < 1$ and represents the complex distortion of the function $f$. Such maps have useful geometric and regularity properties, and provide a natural framework for generalizing conformal maps. They arise naturally in a number of applications, and are closely related with the solutions to elliptic PDEs in the plane. 

In this paper, we will be concerned with the precise degree of regularity and smoothness properties of quasiconformal maps; we will be most interested in determining what H\"older continuity such maps have (that is, which Lipschitz class the functions lie in). A function $f$ defined on an open set $\Omega$ is said to be locally $\alpha$-H\"older continuous if for each compact set $E \subseteq \Omega$ there exists a constant $C = C(f, E)$ with  
$$|f(z_1) - f(z_2)| \le C |z_1 - z_2|^{\alpha}$$ 
for all $z_1, z_2 \in E$; equivalently, $f$ lies in the Lipschitz class $C^{\alpha}(\Omega)$. It is well known that $K$-quasiconformal maps are H\"older continuous with exponent $1/K$, due to an old theorem of Mori \cite{Mor56}. More recently, quantitative upper and lower bounds on the size (in the sense of Hausdorff measure and dimension) of the set where $f$ can attain the worst-case H\"older regularity were given by Astala, Iwaniec, Prause, and Saksman \cite{AstIwaPraSak15bilip} and the author \cite{Bon17qua}.

However, the exponent $1/K$ is not always optimal for a $K$-quasiconformal map. For example, there are bilipschitz $K$-quasiconformal maps defined on $\mathbb{C}$ which are not $(K - \epsilon)$-quasiconformal for any $\epsilon > 0$. A particular case of this is a map exhibiting rotation, such as $z |z|^{i \gamma}$ for an appropriately chosen exponent $\gamma \in \mathbb{R}$. Therefore, it is apparent that the exact H\"older regularity of a quasiconformal map depends on more than just the magnitude of the complex distortion and should instead encode something about the structure of the distortion. A result of Ricciardi \cite{Ric08bel} gave a great deal of information; in that paper, it was shown that if $f$ is a solution to the Beltrami equation $\dzbar f = \mu \dz f$ then $f$ is $\alpha$-H\"older continuous with exponent 
\begin{equation}
\alpha \ge \left(\sup_{S_{\rho, x} \subset \Omega} \frac{1}{|S_{\rho, x}|} \int_{S_{\rho, x}} \frac{|1 - \overline{\eta}^2 \mu|^2}{1 - |\mu|^2} \, d\sigma\right)^{-1} \label{eq:ricciardi_result}
\end{equation}
where $S_{\rho, x}$ is a circle with radius $\rho$ centered at $x \in \Omega$, $\eta$ is an outward unit normal, and $d\sigma$ is the arclength measure. This was proven through a sharp Wirtinger inequality. The integrand here also appeared in \cite{ReiWal65} in the context of ring modules; for more information, as well as some related estimates and theorems about extremizers, see the book \cite{BojGutMarRya13} of Bojarski, Gutlyanskii, Martio, and Ryazanov.

An important application of the regularity results for quasiconformal maps is their connection with solutions to elliptic partial differential equations of the form  
\begin{equation}
\diverg(A \nabla u) = 0 \label{eq:linear_elliptic_pde}
\end{equation} 
where $z \mapsto A(z)$ is an essentially bounded, symmetric, measurable, matrix-valued function satisfying $\lambda \langle \xi, \xi \rangle \le \langle \xi, A(z)\xi\rangle \le \Lambda \langle \xi, \xi \rangle$ for some $0 < \lambda \le \Lambda < \infty$ at almost every $z$. Just as there is a correspondence between the Cauchy-Riemann equation $\dzbar f = 0$ and the Laplacian $\Delta u = 0$ (where $f = u + iv$ and $v$ is the harmonic conjugate of $u$), there is a correspondence between the $\mathbb{C}$-linear Beltrami equation \eqref{eq:c_linear_beltrami} and the divergence form elliptic equation \eqref{eq:linear_elliptic_pde}, where $f = u + iv$ and $v$ is the $A$-harmonic conjugate of $u$. The exact details of this correspondence can be found in, e.g. Chapter 16 of \cite{AstIwaMar09}. 

Solutions to these equations are known to be H\"older continuous, and the study of their regularity has a long history. H\"older continuity of solutions to \eqref{eq:linear_elliptic_pde} was shown by De Giorgi \cite{Deg57}; later, Piccinini and Spagnolo \cite{PicSpa72} gave a quantitative estimate that the H\"older continuity exponent of $u$ is at least $\sqrt{\lambda / \Lambda}$ (as well as further improved bounds for the case of an isotropic matrix $A$). More recently, Ricciardi \cite{Ric08div} showed that the H\"older exponent is at least
$$\alpha \ge \left(\sup_{S_{\rho, x} \subset \Omega} \frac{1}{|S_{\rho, x}|} \int_{S_{\rho, x}} \langle \eta, A \eta\rangle d\sigma\right)^{-1}.$$

The main result of this paper will be an improvement of the H\"older continuity exponent given in \eqref{eq:ricciardi_result} to incorporate an extra term involving the geometry of the underlying map. In particular, we will show that
\begin{theorem}
Let $f : \Omega \to \Omega'$ be a continuous and $W^{1, 2}_{\text{loc}}$ solution to the Beltrami equation $\dzbar f = \mu(z) \dz f$ with $|\mu(z)| \le \frac{K - 1}{K + 1}$. Then $f$ is $\alpha$-H\"older continuous for some exponent $\alpha$ satisfying
$$\alpha \ge \left[4 \pi \sup_{S_{\rho, x} \subset \Omega} \frac{|f(\mathbb{D}_{\rho, x})|}{\mathcal{H}^1\big(f(S_{\rho, x})\big)^2} \sup_{S_{\rho, x} \subset \Omega} \frac{1}{|S_{\rho, x}|} \int \frac{|1 - \bar{\eta}^2 \mu|^2}{1 - |\mu|^2} d\sigma \right]^{-1}$$
where $S_{\rho, x}$ is the circle centered at $x$ with radius $\rho$, $\eta$ is the outward unit normal, and $\sigma$ is arclength measure.
\end{theorem}

Here, it is important to note that the isoperimetric inequality guarantees that 
$$4 \pi \frac{|f(\mathbb{D}_t)|}{\mathcal{H}^1(f(\mathbb{S}_t))^2} \le 1$$
for all $t$ (and is frequently strictly less than $1$); here, $\mathbb{D}_t$ is the disk centered at the origin with radius $t$. Our result therefore gives an improvement over the previously known regularity whenever we can impose an upper bound on $4 \pi \frac{|f(\mathbb{D}_t)|}{\mathcal{H}^1(f(\mathbb{S}_t))^2}$; for example, any affine map which stretches differently in two orthogonal directions will exhibit this. Furthermore, we can use this information to determine the structure of the extremizers for H\"older continuity. We have the following definition:
\begin{defn}\label{defn:extremizer}
Let $f$ be a $K$-quasiconformal map, a $K$-quasiregular map, or a solution to \eqref{eq:linear_elliptic_pde} with ellipticity constants satisfying $\sqrt{\Lambda / \lambda} = K$. We say that $f$ is an extremizer for H\"older continuity at the origin if $f$ is not more than $1/K$-H\"older continuous there. In particular, for each $\epsilon > 0$, there is a sequence $r_n \to 0$ such that $|f(r_n) - f(0)| \ge r_n^{1/K + \epsilon}$. 
\end{defn}

We can now give the form of the complex distortion of a $K$-quasiconformal map that exhibits the worst-case regularity. Motivated by the fact that the Beltrami coefficient of the radial stretch $z|z|^{1/K - 1}$ is $-k z/\overline{z}$ with $k = \frac{K + 1}{K - 1}$, we have the following result:
\begin{theorem} 
Suppose that $f$ is $K$-quasiconformal and an extremizer for H\"older continuity at the origin. Write the Beltrami coefficient in the form $\mu(z) = \frac{z}{\overline{z}} \left(-k + \epsilon(z)\right)$. Then there is a sequence of scales $t_n \to 0$ for which
$$\int_{t_n}^1 \frac 1 r \left(\frac 1 {|S_r|} \int_{S_r} \real \epsilon(z) \, d\sigma \right) \, dr = o\left(\log \frac 1 {t_n}\right).$$
\end{theorem}
There is an analogous theorem for the geometric distortion properties of such extremizers, and how far they can be from a map which preserves circularity: 
\begin{theorem}
Suppose that $f$ is $K$-quasiconformal and an extremizer for H\"older continuity at the origin. Define a function $\delta$ by $\frac{|f(\mathbb{D}_t)|}{\mathcal{H}^1(f(S_t))^2} = \frac{1}{4\pi (1 + \delta(t))}.$
Then there is a sequence $t_n \to 0^+$ such that
$$\int_{t_n}^1 \frac{\delta(r)}{r} \, dr = o\left(\log \frac 1 {t_n}\right).$$
\end{theorem}
Our last main results are on the regularity and H\"older continuity extremizers of solutions to the elliptic equation \eqref{eq:linear_elliptic_pde}. Suppose $u \in W^{1, 2}_{\text{loc}}(\Omega)$ is a continuous solution to $\diverg(A \nabla u) = 0$ on a simply connected domain $\Omega$, that $A$ is symmetric, measurable, and satisfies the ellipticity bound $\frac 1 K |\xi|^2 \le \langle \xi, A(z) \xi \rangle \le K |\xi|^2$ almost everywhere. Let $v$ be the $A$-harmonic conjugate of $u$ and $f = u + iv$. We will show that
\begin{theorem}
If $u$ is an extremizer for H\"older continuity at the origin, $u$ must be of the form
$$u = \Phi \circ g$$
where $\Phi$ is harmonic with non-vanishing gradient at $g(0)$, $g$ is $K$-quasiconformal, and $g$ exhibits the worst-case regularity for a $K$-quasiconformal map. In particular, the bounds of the previous two theorems apply to the Beltrami coefficient and circular distortion of $g$.
\end{theorem}

The outline of this paper is as follows. In Section 2, we prove the main theorem on H\"older continuity of quasiconformal maps. In Section 3, we use this information to classify the extremizers and study geometric distortion, and extend the results to quasiregular maps. In Section 4, we apply the results of the previous sections to solutions to elliptic PDEs.

%
%

\section{Estimate of H\"older Exponent}
The main result of this section is estimate the exact degree of H\"older continuity of a $K$-quasiconformal map through the behavior of the associated Beltrami coefficient. 

\begin{theorem} \label{thm:main_holder}
Let $f : \Omega \to \Omega'$ be a continuous and $W^{1, 2}_{\text{loc}}(\Omega)$ solution to the Beltrami equation $\dzbar f = \mu(z) \dz f$ with $|\mu(z)| \le \frac{K - 1}{K + 1}$. Then $f$ is $\alpha$-H\"older continuous for some exponent $\alpha$ satisfying
$$\alpha \ge \left[4 \pi \sup_{S_{\rho, x} \subset \Omega} \frac{|f(\mathbb{D}_{\rho, x})|}{\mathcal{H}^1\big(f(S_{\rho, x})\big)^2} \sup_{S_{\rho, x} \subset \Omega} \frac{1}{|S_{\rho, x}|} \int_{S_{\rho, x}} \frac{|1 - \bar{\eta}^2 \mu|^2}{1 - |\mu|^2} d\sigma \right]^{-1}$$
where $S_{\rho, x}$ is the circle centered at $x$ with radius $\rho$, $\eta$ is the outward unit normal, and $\sigma$ is arclength measure.
\end{theorem}

Here, the suprema can be regarded as the essential supremum over the radii $\rho$ for each point $x$. Note that the isoperimetric inequality guarantees that 
$$4 \pi \sup_{S_{\rho, x} \subset \Omega} \frac{|f(\mathbb{D}_{\rho, x})|}{\mathcal{H}^1\big(f(S_{\rho, x})\big)^2} \le 1$$
and so we recover the result \eqref{eq:ricciardi_result} of Ricciardi for the case of the homogeneous Beltrami equation; moreover, 
\begin{align*}
\int_{S_{\rho, x}} \frac{|1 - \bar{\eta}^2 \mu|^2}{1 - |\mu|^2} d\sigma &\le \int_{S_{\rho, x}} \frac{(1 + |\mu|)^2}{1 - |\mu|^2} d\sigma = \int_{S_{\rho, x}} \frac{1 + |\mu|}{1 - |\mu|} d\sigma \le K |S_{\rho, x}|
\end{align*}
for almost every circle (that is, for almost every positive radius). This recovers the classic exponent of $1/K$.

The theorem also shows that H\"older continuity of a quasiconformal map is locally determined by the structure of the Beltrami coefficient. For example, if there is an open set $E$ where $\|\mu \chi_E\|_{\infty} < \|\mu\|_{\infty}$, then the quasiconformal map displays better-than-expected H\"older continuity on the entirety of the open set. As will be proved later on, this idea also gives powerful constraints on the complex distortion of a map which has the worst-case H\"older continuity (even at a single point). It is worth mentioning that these results also follow from Stoilow factorization. We now turn to the proof of Theorem \ref{thm:main_holder}.

\begin{proof}
Without loss of generality, we will look at circles centered at the origin. Our starting point will be an adaptation of a classical argument of Morrey \cite{Mor38}. To this end, define $\varphi(t) = \int_{\disk_t} J_f = |f(\disk_t)|$. If we can show that $\varphi(t) \le \varphi(1) t^{2c}$, then quasisymmetry shows that the worst case length distortion is controlled by
$$|f(te^{i\theta}) - f(0)|^2 \sim_K |f(\disk_t)| \le \varphi(1) t^{2c}$$
which implies a H\"older exponent no worse than $c$. Our task is therefore to estimate $\varphi$, which we will do by controlling $\varphi$ by its derivative. 

In order to do this, we will compute the circumference of the quasicircle $f(S_t)$ explicitly, where $S_t$ has radius $t$ and is centered at the origin. Parameterize the quasicircle by $\gamma(\theta) = f(te^{i\theta})$ for $\theta \in [0, 2\pi]$; note that for almost every $t$, the quasicircle has positive and finite length. Indeed, since $f \in W^{1, 2}_{\text{loc}}$, $f$ is absolutely continuous on the circle $S_t$ for almost every $t \in [0, \infty)$. Likewise, since the Beltrami coefficient $\mu$ is defined almost everywhere with respect to the area measure, Fubini's theorem guarantees that $\mu$ is defined at almost every point (with respect to arclength) on almost every circle. We also have that $f_{\zbar} = \mu f_z$ almost everywhere in the plane, so almost everywhere on almost every circle. Thus, we can compute the length by
$$\operatorname{Length} = \int_0^{2\pi} \left|\frac{d}{d\theta} \gamma(\theta)\right| \, d\theta.$$
Writing $f = u + iv$, we have that
\begin{align}
\frac 1 {t^2} \left|\frac{d}{d\theta} \gamma\right|^2 &= \frac 1 {t^2} \left|\frac{d}{d\theta} \left[u(t \cos \theta, t \sin \theta) + i v(t \cos \theta, t \sin \theta)\right]\right|^2 \nonumber \\
&= \frac 1 {t^2} \left|-t u_x \sin \theta + t u_y \cos \theta + i \big[-t v_x \sin \theta + t v_y \cos \theta\big]\right|^2 \nonumber \\
&= u_x^2 \sin^2 \theta + u_y^2 \cos^2 \theta - 2u_x u_y \sin \theta \cos \theta \\
&\quad\quad + v_x^2 \sin^2 \theta + v_y^2 \cos^2 \theta - 2 v_x v_y \sin \theta \cos \theta \nonumber \\
&= |f_x|^2 \sin^2 \theta + |f_y|^2 \cos^2 \theta - 2 \sin \theta \cos \theta \big(u_x u_y + v_x v_y\big) \label{eq:length_squared}
\end{align}
If we write the Beltrami equation in terms of $x-$ and $y-$derivatives rather than the Wirtinger derivatives, we find that
\begin{equation}\label{eq:second_term}
f_x + i f_y = \mu (f_x - i f_y) \implies f_x (1 - \mu) = -i f_y (1 + \mu) \implies f_y = i \frac{1 - \mu}{1 + \mu} f_x.
\end{equation}
Furthermore,
$$\real (\overline{f_x} f_y) = \real\big((u_x - i v_x)(u_y + i v_y)\big) = u_x u_y + v_x v_y.$$
Thus, the final term in \eqref{eq:length_squared} can be replaced with
\begin{align}
u_x u_y + v_x v_y &= \real (\overline{f_x} f_y) \nonumber \\
&= \real \left(\overline{f_x} i \frac{1 - \mu}{1 + \mu} f_x\right) \nonumber \\
&= |f_x|^2 \left(-\imag \frac{1 - \mu}{1 + \mu}\right) \nonumber \\
&= -\frac{|f_x|^2}{|1 + \mu|^2} \imag\big((1 - \mu)(1 + \overline{\mu})\big) \nonumber \\
&= -\frac{|f_x|^2}{|1 + \mu|^2} \imag(\overline{\mu} - \mu) \nonumber \\
&= \frac{2|f_x|^2 \imag \mu}{|1 + \mu|^2} \label{eq:third_term}
\end{align}
We also wish to rewrite $f_x$ in terms of the Jacobian, so as to relate $\varphi$ to $\varphi'$. We have
$$J_f = |f_z|^2 - |f_{\zbar}|^2 = (1 - |\mu|^2) |f_z|^2.$$
On the other hand, $f_x = f_z + f_{\zbar} = (1 + \mu) f_z$, so that
\begin{equation}\label{eq:jacobian_fx}
J_f = \frac{1 - |\mu|^2}{|1 + \mu|^2} |f_x|^2
\end{equation}
Combining the equations (\ref{eq:length_squared})-(\ref{eq:jacobian_fx}), we find that the length $\mathcal{H}^1\big(f(S_t)\big)$ of the quasicircle $f(S_t)$ can be written as
\begin{align}
&\int_0^{2\pi} \frac{|1 + \mu|}{\sqrt{1 - |\mu|^2}} J_f^{1/2} \sqrt{\sin^2 \theta + \left|\frac{1 - \mu}{1 + \mu}\right|^2 \cos^2 \theta - 4 \sin \theta \cos \theta \frac{\imag \mu}{|1 + \mu|^2}} \, t d\theta \nonumber \\
&= \int_0^{2\pi} \frac{J_f^{1/2}}{\sqrt{1 - |\mu|^2}} \sqrt{|1 + \mu|^2 \sin^2 \theta + |1 - \mu|^2 \cos^2 \theta - 4 \sin \theta \cos \theta \imag \mu} \, t d\theta. \label{eq:sqrt_jacobian}
\end{align}
It remains to simplify the term within the square root. We can expand it to find that 
\begin{align}
|1 + \mu|^2 \sin^2 \theta + &|1 - \mu|^2 \cos^2 \theta - 4 \sin \theta \cos \theta \imag \mu  \nonumber \\
&= 1 + |\mu|^2 + 2 \real \mu \sin^2 \theta - 2 \real \mu \cos^2 \theta - 4 \sin \theta \cos \theta \imag \mu \nonumber \\
&= 1 + |\mu|^2 - 2 \real \mu \cos 2 \theta - 2 \imag \mu \sin 2 \theta \nonumber \\
&= 1 + |\mu|^2 - 2 \real \big((\cos 2\theta - i \sin 2 \theta)(\real \mu + i \imag \mu)\big) \nonumber \\
&= 1 + |\mu|^2 - 2 \real(e^{-2i\theta} \mu) \nonumber \\
&= |1 - e^{-2i\theta} \mu|^2
\end{align}
Noting that $e^{i\theta} = \eta$ is the outer normal from the circle, combining this with (\ref{eq:sqrt_jacobian}) we arrive at
\begin{equation}\label{eq:final_length}
\mathcal{H}^1\big(f(S_t)\big) = \int_0^{2\pi} \left(\frac{|1 - \overline{\eta}^2 \mu|^2}{1 - |\mu|^2}\right)^{1/2} J_f^{1/2} \, t \, d\theta.
\end{equation}
We are now ready to make the estimate of $\varphi$. Denote
$$A = 4 \pi \sup_{t} \frac{|f(\mathbb{D}_{t})|}{\mathcal{H}^1\big(f(S_{t})\big)^2} \quad \text{ and } \quad C = \sup_t \frac{1}{|S_t|} \int_{S_t} \frac{|1 - \overline{\eta}^2 \mu|^2}{1 - |\mu|^2}\, d\sigma. $$
recalling that $d\sigma = t d\theta$ is arclength. We then have for almost every $t$ that
\begin{align}
\varphi(t) &= \frac{|f(\disk_t)|}{\mathcal{H}^1\big(f(S_t)\big)^2} \mathcal{H}^1\big(f(S_t)\big)^2 \nonumber \\
&\le \frac{1}{4\pi} A \left(\int_0^{2\pi} \left(\frac{|1 - \overline{\eta}^2 \mu|^2}{1 - |\mu|^2}\right)^{1/2} J_f^{1/2} \, t \, d\theta\right)^2 \nonumber \\
&\le \frac{1}{4\pi} A \int_0^{2\pi} \frac{|1 - \overline{\eta}^2 \mu|^2}{1 - |\mu|^2} \, t\,  d\theta \int_0^{2\pi} J_f \, t \, d\theta \nonumber \\
&= \frac 1 2 At \left(\frac{1}{2 \pi t} \int_{S_t} \frac{|1 - \overline{\eta}^2 \mu|^2}{1 - |\mu|^2}\, d\sigma \right) \left(\int_{0}^{2\pi} J_f \, t \, d\theta \right) \nonumber \\
&\le \frac 1 2 AC t \varphi'(t) \label{eq:big_equation}
\end{align}
for almost every $t$. We then find that
$$\frac{d}{dt} \left[t^{-2/AC} \varphi(t)\right] = t^{-2/AC - 1} \left[-\frac{2}{AC} \varphi(t) + t \varphi'(t)\right] \ge 0.$$
almost everywhere. Integrating this inequality over $[t, 1]$ leads to $\varphi(t) \le \varphi(1) t^{2/AC}$, which is the desired result.
\end{proof}

%
%

\section{Extremizers for H\"older Continuity}
Here we will study the structure of the Beltrami equation for the extremizers of H\"older continuity. Recall that by Mori's Theorem, a $K$-quasiconformal map is at least $\frac 1 K$-H\"older continuous. We will show that, in some sense, the extremizers must have Beltrami coefficients which are very close to the coefficient for a pure radial stretch. Denote $k = \frac{K - 1}{K + 1}$; since the Beltrami coefficient has $|\mu| \le k$, we can always write
\begin{align}
\mu(z) = e^{2i \theta} \left(-k + \epsilon(z)\right) \label{eq:coefficient_format}
\end{align}
with $\theta$ being the argument of $z$, and $\epsilon$ some function such that has nonnegative real part. Note that $-k e^{2i\theta}$ is precisely the Beltrami coefficient of the radial stretch $z|z|^{1/K - 1}$. We now have our result:

\begin{theorem}\label{thm:extremizer_beltrami}
Suppose that $f$ is $K$-quasiconformal and an extremizer for H\"older continuity at the origin. Write the Beltrami coefficient in the form \eqref{eq:coefficient_format}. Then there is a sequence of scales $t_n \to 0$ for which
$$\int_{t_n}^1 \frac 1 r \int_{S_r} \real \epsilon(z) \, d\tau \, dr = o\left(\log \frac 1 {t_n}\right)$$
where $d\tau = \frac{d\sigma}{2\pi r}$ is normalized arclength.
\end{theorem}

Morally, this says that the circular averages of $\real \epsilon(z)$ are tending to zero in some quantitative sense - otherwise, the integral $\int_{t_n}^1 \frac{dr}{r}$ would give some non-zero fraction of $\log 1/t_n$. 

\begin{proof}
We proceed through two steps; the first is to estimate the impact that our perturbation by $\epsilon$ has on
$$g(t) := \int_{S_r} \frac{|1 - e^{-2i\theta} \mu|^2}{1 - |\mu|^2} \, d\tau$$
and the second is to sharpen the estimate of $|f(\disk_t)|$ accordingly. Recall that Theorem \ref{thm:main_holder} (and in particular equation \eqref{eq:big_equation} with the estimate $A \le 1$) tells us that
$$\varphi(r) \le \frac {r g(r)}{2} \varphi'(r)$$
for almost every $r \in [0, 1]$. Integrating this differential inequality, we find that
\begin{align}
\ln \frac{\varphi(1)}{\varphi(t)} = \int_t^1 \frac{\varphi'(r)}{\varphi(r)}\, dr &\ge \int_t^1 \frac{2}{rg(r)} \, dr \label{eq:estimate_phi}
\end{align}
Later on, we will prove that there is a constant $c_1 > 0$ such that 
\begin{align}
g(r) \le K - c_1 \int_{S_r} \real \epsilon \, d\tau \label{eq:estimate_g}
\end{align}
so that
$$\frac 1 {g(r)} \ge \frac 1 K \cdot \frac{1}{1 - (c_1/K) \int_{S_r} \real \epsilon \, d\tau} \ge \frac 1 K \left[1 + c \int_{S_r} \real \epsilon \, d\tau\right]$$
with $c = c_1 / K > 0$. With this estimate in mind, we can continue \eqref{eq:estimate_phi} to find that
\begin{align}
\ln \frac{\varphi(1)}{\varphi(t)} &\ge \int_t^1 \frac 2 r \cdot \frac 1 K \left[1 + c \int_{S_r} \real \epsilon \, d\tau\right] \, dr \nonumber \\
&= \ln t^{-2/K} + \frac{2c}{K} \int_t^1 \frac 1 r \int_{S_r} \real \epsilon \, d\tau dr
\end{align}
Rearranging this gives us
\begin{align}
\varphi(t) \le \varphi(1) t^{2/K} \exp \left(-2\frac{c}{K} \int_t^1 \frac 1 r \int_{S_r} \real \epsilon \, d\tau dr\right)
\end{align}
Now if $f$ is no more regular than $1/K$-H\"older continuous and $\gamma > 0$, there is a sequence of scales $r_n \to 0$ (depending on $\gamma$) for which $\varphi(r_n) \ge r_n^{2/K + \gamma} \varphi(1)$ for all $n$. We therefore have
\begin{align}
r_n^{2/K + \gamma} \le r_n^{2/K} \exp \left(-2\frac{c}{K} \int_{r_n}^1 \frac 1 r \int_{S_r} \real \epsilon \, d\tau dr\right)
\end{align}
which is the key estimate behind our constraint on the structure of $\real \epsilon$. Equivalently,
\begin{align}
\gamma \log r_n &\le - 2 \frac{c}{K} \int_{r_n}^1 \frac 1 r \int_{S_r} \real \epsilon \, d\tau dr \nonumber \\
\implies \frac{\gamma K}{2c} \log \frac 1 {r_n} &\ge \int_{r_n}^1 \frac 1 r \int_{S_r} \real \epsilon \, d\tau dr
\end{align}
Taking a sequence of $\gamma_m \to 0$ and choosing scales $t_m$ appropriately gives the desired estimate.

All that remains now is to show \eqref{eq:estimate_g}. Writing $\mu = e^{2i\theta} (-k + \real \epsilon)$ and choosing $w = -k + \real \epsilon$ we have that
\begin{align}
\frac{|1 - e^{-2i\theta} \mu|^2}{1 - |\mu|^2} &= \frac{|1 - w|^2}{1 - |w|^2}. \nonumber
\end{align}
We will show that
\begin{align}
\frac{|1 - w|^2}{1 - |w|^2} - \frac{1 + k}{1 - k} \lesssim -\real \epsilon
\end{align}
where the implied constant only depends on $k$; the estimate \eqref{eq:estimate_g} follows immediately from integrating this over $S_r$, recalling that $\frac{1 + k}{1 - k} = K$ and that $d\tau$ is a probability measure. It is worth mentioning that the estimate $|1 - w|^2 / (1 - |w|^2) - (1 + k) / (1 - k) \le 0$ is immediate from the triangle inequality, but that we need to sharpen it a little bit. To carry out this estimate, observe that since $w$ is real we have
\begin{align}
\frac{|1 - w|^2}{1 - |w|^2} - \frac{1 + k}{1 - k} &= \frac{1 - w}{1 + w} - \frac{1 + k}{1 - k} \nonumber \\
&= \frac{(1 - w)(1 - k) - (1 + k)(1 + w)}{(1 + w)(1 - k)} \nonumber \\
&= \frac{-2(w + k)}{(1 + w)(1 - k)} \nonumber \\
&= \frac{-2 \real \epsilon}{(1 + w)(1 - k)} \label{eq:simplified_w_eps} 
\end{align}
Now $w \le k$ since $\epsilon(z)$ lies within the disk centered at $k$ with radius $k$, and so we have
$$\frac{|1 - w|^2}{1 - |w|^2} - \frac{1 + k}{1 - k} \le \frac{-2 \real \epsilon}{(1 + k)(1 - k)}$$
which proves \eqref{eq:estimate_g}.
\end{proof}

In a similar manner, we can study the local distortion properties of an extremizer. Infinitesimally, a quasiconformal map must take circles to ellipses with bounded eccentricity; we will show that (in an appropriate sense), the extremizers for H\"older continuity must almost take circles to circles. Our approach to this is similar to the previous theorem: we will write $|f(\mathbb{D}_t)| / \mathcal{H}^1(f(S_t))^2$ as a perturbation of $1/4\pi$ and control the perturbation.

\begin{theorem}\label{thm:extremizer_distortion}
Suppose that $f$ is $K$-quasiconformal and an extremizer for H\"older continuity at the origin. Define functions $h$ and $\delta$ by 
$$h(t) := \frac{|f(\mathbb{D}_t)|}{\mathcal{H}^1(f(S_t))^2} = \frac{1}{4\pi (1 + \delta(t))}.$$
Then there is a sequence $t_n \to 0^+$ such that
$$\int_{t_n}^1 \frac{\delta(r)}{r} \, dr = o\left(\log \frac 1 {t_n}\right).$$
\end{theorem}

\begin{proof}
As a first remark, the isoperimetric inequality shows that $h(t) \le \frac 1 {4\pi}$ for all $t$, so $\delta(t) \ge 0$ everywhere; also, $\delta(t) < \infty$ almost everywhere. In analogy with the previous theorem, we can conclude from \eqref{eq:big_equation} that
$$\varphi(t) \le 2\pi K t h(t) \varphi'(t).$$
Rearranging and integrating leads to
\begin{align}
\varphi(t) \le \varphi(1) \exp \left(-\int_{t}^1 \frac{dr}{2\pi Kr h(r)}\right) \label{eq:upper_bound_distortion_integral}
\end{align}
Now since $f$ is an extremizer for H\"older continuity, for any sequence $\gamma_n \to 0^+$ there is a sequence of scales $t_n \to 0^+$ for which $\varphi(t_n) \ge \varphi(1) t_n^{2/K (1 + \gamma_n)}$; combining this with \eqref{eq:upper_bound_distortion_integral} leads to
$$t_n^{\frac 2 K (1 + \gamma_n)} \le \exp \left(-\int_{t_n}^1 \frac{dr}{2\pi Kr h(r)}\right).$$
Taking a logarithm and rearranging, we find that
\begin{align}
\frac 2 K (1 + \gamma_n) \log t_n &\le - \int_{t_n}^1 \frac{dr}{2\pi K r h(r)}\nonumber \\
&= - \int_{t_n}^1 \frac{dr}{2\pi K r \frac{1}{4\pi(1 + \delta(r))}}\nonumber \\
&= - \frac 2 K \int_{t_n}^1 \frac{1 + \delta(r)}{r} \, dr \nonumber \\
&= \frac 2 K \log t_n - \frac 2 K \int_{t_n}^1 \frac{\delta(r)}{r} \, dr. \label{eq:distortion_integral_simplified}
\end{align}
Rearranging this leads to
$$\int_{t_n}^1 \frac{\delta(r)}{r} \, dr \le \gamma_n \log \frac 1 {t_n}$$
as desired.
\end{proof}

\begin{corl}\label{corl:extremizer_distortion_density}
Suppose $f$ is an extremizer for H\"older continuity at the origin and $\delta$ is defined as in Theorem \ref{thm:extremizer_distortion}. Then for any $\delta_0 > 0$, the set $\{r : \delta(r) > \delta_0\}$ has zero lower density at $0$.
\end{corl}

\begin{proof}
Fix $\delta_0 > 0$ and suppose, intending a contradiction, that the lower density of $\{r : \delta(r) > \delta_0\}$ had lower density at least $\eta > 0$ at zero. Then there exists a scale $\epsilon > 0$ such that for all $\gamma < \epsilon$,
$$\frac{|\{r : \delta(r) > \delta_0\} \cap [0, \gamma]\}|}{\gamma} > \frac{\eta}{2}.$$
Consequently, there exists an $M$ depending only on $\eta$ such that for all $\gamma < \epsilon$,
$$\frac{|\{r : \delta(r) > \delta_0\} \cap [\gamma/M, \gamma]\}|}{\gamma} > \frac{\eta}{100}.$$
Therefore, we can estimate that
\begin{align}
\int_{\gamma/M}^{\gamma} \frac{\delta(r)}{r} \, dr &\ge \frac{\delta_0}{\gamma} \left|[\gamma/M, \gamma] \cap \{r : \delta(r) > \delta_0\}\right| \nonumber \\
&> \frac{\delta_0}{\gamma} \cdot \frac{\gamma \eta}{100} \nonumber \\
&= \frac{\delta_0 \eta}{100} \label{eq:single_scale}
\end{align}
Summing \eqref{eq:single_scale} over intervals $[\gamma/M, \gamma], [\gamma/M^2, \gamma/M]$ and so on, it is immediate that 
$$\int_t^1 \frac{\delta(r)}{r} \, dr \gtrsim \log \frac 1 t$$
for all sufficiently small $t$. This contradicts the result of Theorem \ref{thm:extremizer_distortion} and the corollary follows.
\end{proof}

Finally, as usual, there is a natural generalization of the quasiconformal result to the quasiregular result.
\begin{theorem}\label{thm:quasiregular_extremizers}
A $K$-quasiregular map $g$ is an extremizer for H\"older continuity at the origin if and only if it is of the form $g = \Phi \circ f$, where $f$ is $K$-quasiconformal and an extremizer for H\"older continuity at the origin, and $\Phi$ is conformal in a neighborhood of $f(0)$.
\end{theorem}

\begin{proof}
Suppose $g$ is a $K$-quasiregular extremizer. By the Stoilow factorization theorem, there exists a holomorphic $\Phi$ and $K$-quasiconformal $f$ such that 
$$g = \Phi \circ f.$$
Since $\Phi$ is a smooth function, $\Phi \circ f$ is at least as regular as $f$ is (in particular, if $f$ is $\alpha$-H\"older continuous, then so is $g$); but since $g$ is not more than $1/K$-H\"older continuous, neither is $f$. Furthermore, we must have $\Phi'(f(0)) \ne 0$ (which implies conformality); otherwise, $\Phi(f(z)) - \Phi(f(0))$ vanishes to at least second order at $0$, and $\Phi \circ f$ is H\"older continuous at $0$ with exponent at least $\min\{1, 2/K\} > 1/K$.

On the other hand, if $g = \Phi \circ f$ with $f$ an extremizer and $\Phi$ conformal at $f(0)$, we can locally invert $\Phi$ as a smooth function, so that $f = \Phi^{-1} \circ g$. Hence $f$ is at least as regular as $g$ is; but since $f$ has the worst-case regularity, so must $g$.
\end{proof}
%
%

\section{Applications to Elliptic PDEs}
Next, we will use the relationship between quasiconformal maps and solutions to elliptic partial differential equations in order to deduce regularity results and classify the extremizers for H\"older continuity. The starting point for this work is the correspondence laid down in \cite{AstIwaMar09}, Chapter 16. Throughout, we will assume that $A(z)$ is a matrix valued function which is measurable, symmetric, and satisfies the ellipticity bound 
$$\frac 1 K |\xi|^2 \le \langle A(z) \xi, \xi \rangle \le K |\xi|^2$$
at almost every $z \in \Omega$. Note that this implies that $A(z)$ is positive definite, and an equivalent formulation is the unified inequality 
$$|\xi|^2 + |A(z) \xi|^2 \le \left(K + \frac 1 K\right)\langle A(z) \xi, \xi\rangle.$$
Consider the divergence form equation
\begin{equation}
\diverg A(z) \nabla u = 0. \label{eq:divergence_form}
\end{equation}
If $\Omega$ is a simply connected domain and $u \in W^{1, 2}_{\text{loc}}(\Omega)$ is a solution to \eqref{eq:divergence_form}, the Poincar\'e lemma guarantees that there exists an $A$-harmonic conjugate $v$; that is, $v \in W^{1, 2}_{\text{loc}}(\Omega)$ solves
$$\nabla v = \ast A(z) \nabla u$$
where $\ast$ is the Hodge star operator, viewed as the matrix
$$\ast = \left[\begin{array}{cc} 0 & -1 \\ 1 & 0 \end{array}\right].$$
Define $f = u + iv$; we now claim that the ellipticity condition implies that $f$ is $K$-quasiregular. Following Theorem \ref{thm:quasiregular_extremizers}, we will be able to use this to determine the regularity and extremizers for H\"older continuity. To see that $f$ is actually quasiregular, we may compute that
$$\|Df\|^2 = |\nabla u|^2 + |\nabla v|^2 = |\nabla u|^2 + |\ast A(z) \nabla u|^2 \le \left(K + \frac 1 K\right) \langle A(z) \nabla u, \nabla u\rangle.$$
It is immediate to check that $\langle A(z) \nabla u, \nabla u\rangle = J(z, f)$ is the Jacobian of $f$, and we therefore have $\|Df\|^2 \le (K + \frac 1 K) J(z, f)$. Rewriting the Hilbert-Schmidt norm in terms of Wirtinger derivatives, this implies that $|\dzbar f| \le \frac{K - 1}{K + 1} |\dz f|$ as desired. This brings us to the first theorem on regularity, which is a direct application of Theorem \ref{thm:main_holder}:

\begin{theorem}\label{thm:regularity_elliptic}
Let $u \in W^{1, 2}_{\text{loc}}(\Omega)$ be a continuous $W^{1, 2}_{\text{loc}}(\Omega)$ solution to \eqref{eq:divergence_form} on a simply connected domain $\Omega$, where $v$ is its $A$-harmonic conjugate, and $f = u + iv$. Let $\mu(z)$ denote the complex distortion of $f$. Then $u$ is $\alpha$-H\"older continuous with some exponent $\alpha$, where
$$\alpha \ge \left[4 \pi \sup_{S_{\rho, x} \subset \Omega} \frac{|f(\mathbb{D}_{\rho, x})|}{\mathcal{H}^1\big(f(S_{\rho, x})\big)^2} \sup_{S_{\rho, x} \subset \Omega} \frac 1 {|S_{\rho, x}|} \int \frac{|1 - \bar{\eta}^2 \mu|^2}{1 - |\mu|^2} d\sigma \right]^{-1}.$$
In particular, $u = \real f$ is H\"older continuous with exponent at least $1/K$.
\end{theorem}

In a similar manner, we can find the extremizers for H\"older continuity.
\begin{theorem}\label{thm:elliptic_extremizer}
With the assumptions and notation of Theorem \ref{thm:regularity_elliptic}, suppose that $u$ is an extremizer for H\"older continuity at the origin. Then there is a harmonic map $\Phi$ and a $K$-quasiconformal map $g$ such that $g$ is an extremizer for H\"older continuity at the origin, $\Phi$ has non-vanishing gradient in a neighborhood of $g(0)$, and $f = \Phi \circ g$. In particular, the Beltrami coefficient and circular distortion of $g$ satisfy the bounds of Theorems \ref{thm:extremizer_beltrami} and \ref{thm:extremizer_distortion} respectively.
\end{theorem}

\begin{proof}
This is essentially the same idea as the proof of Theorem \ref{thm:quasiregular_extremizers}. Since $u$ is the real part of a $K$-quasiregular map $f$, we can use Stoilow factorization to write $u = (\real \Psi) \circ g$ with $\Psi$ holomorphic and $g$ being $K$-quasiconformal. As before, $\Psi$ must have non-vanishing gradient at $g(0)$ (so as not to improve the regularity), and $g$ must be an extremizer for H\"older continuity at the origin; the result follows.
\end{proof}

Finally, we also have a generalization of the result on extremizers to nonlinear elliptic PDEs.
\begin{theorem}\label{thm:nonlinear_extremizer}
Let $\Omega \subseteq \mathbb{C}$ be a simply connected domain and suppose $\mathcal{A} : \Omega \times \mathbb{C} \to \mathbb{C}$ is measurable in $z \in \Omega$ and continuous in $\xi \in \mathbb{C}$ and satisfies the ellipticity condition
$$|\xi|^2 + |\mathcal{A}(z, \xi)|^2 \le \left(K + \frac 1 K\right) \langle \xi, \mathcal{A}(z, \xi)\rangle.$$
Then if $u \in W^{1, 2}_{\text{loc}}(\Omega)$ is a solution to
\begin{equation}
\diverg \mathcal{A}(z, \nabla u) = 0 \label{eq:nonlinear_elliptic}
\end{equation}
Then if $u$ is an extremizer for H\"older continuity at the origin, $u$ is the form of Theorem \ref{thm:elliptic_extremizer}.
\end{theorem}

\begin{proof}
By Theorem 16.1.8 of \cite{AstIwaMar09}, every solution $u \in W^{1, 2}_{\text{loc}}(\Omega)$ of \eqref{eq:nonlinear_elliptic} solves a linear elliptic equation
$$\diverg \mathbb{A}(z) \nabla u = 0$$
with $\mathbb{A}(z)$ a positive definite, symmetric measurable matrix field of determinant $1$. Furthermore, we have the ellipticity bound
$\frac 1 K |\xi|^2 \le \langle \mathbb{A}(z) \xi, \xi\rangle \le K |\xi|^2$
and the result follows from Theorem \ref{thm:elliptic_extremizer}.
\end{proof}

\section{References}

\begingroup
\renewcommand{\section}[2]{}%
\bibliographystyle{plain}
\bibliography{dissertation_bibliography.bib}
\endgroup

\end{document}